\pdfoutput=1

\documentclass[lettersize]{article}
\usepackage{sampta,amsmath,epsfig,fancyhdr,amsthm,paralist}
\usepackage[numbers]{natbib}

\theoremstyle{plain}
\newtheorem{lemma}{Lemma}[section]
\newtheorem{proposition}[lemma]{Proposition}
\newtheorem{theorem}[lemma]{Theorem}
\newtheorem{corollary}[lemma]{Corollary}

\theoremstyle{definition}
\newtheorem{definition}[lemma]{Definition}

\theoremstyle{remark}
\newtheorem{remark}[lemma]{Remark}

\DeclareMathOperator{\divergence}{div}

\DeclareMathOperator{\TV}{TV}

\DeclareMathOperator{\BV}{BV}
\DeclareMathOperator{\BD}{BD}

\DeclareMathOperator{\TGV}{TGV}

\DeclareMathOperator{\BGV}{BGV}

\DeclareMathOperator{\Sym}{Sym}

\DeclareMathOperator{\dom}{dom}

\newcommand{\doublehookrightarrow}%
{\DOTSB\lhook\joinrel\relbar\!\!\!\!\lhook\joinrel\rightarrow}

\newcommand{\longrightharpoonup}%
{\relbar\joinrel\rightharpoonup}

\newcommand{\conditionalcomma}[1]{\ifx#1\empty\else,\fi}

\newcommand{\RR}{\mathrm{I\mspace{-2.5mu}R}}
\newcommand{\NN}{\mathrm{I\mspace{-2.5mu}N}}

\newcommand{\set}[2]{\{{#1} \ \bigl| \ {#2}\}}

\newcommand{\sett}[1]{\{{#1}\}}

\newcommand{\placeholder}{\,\cdot\,}

\newcommand{\grad}{\nabla}

\newcommand{\abs}[1]{{|{#1}|}}

\newcommand{\inprod}{\cdot}

\newcommand{\norm}[2][]{\|{#2}\|_{#1}}

\newcommand{\dd}[1]{\ \mathrm{d}{#1}}
\newcommand{\conv}{\ast}
\newcommand{\transp}{\mathrm{T}}
\newcommand{\wrightarrow}{\rightharpoonup}

\newcommand{\compose}{\circ}

\newcommand{\seq}[1]{\{{#1}\}}

\newcommand{\Cspace}[3][]{\mathcal{C}_{#1}^{#2}({#3})}
\newcommand{\Ccspace}[2]{\mathcal{C}_{\mathrm{c}}^{#1}({#2})}

\newcommand{\lebesgueL}[1]{L^{#1}}

\newcommand{\LPspace}[2]{\lebesgueL{#1}({#2})}
\newcommand{\LPlocspace}[2]{\lebesgueL{#1}_{\mathrm{loc}}({#2})}

\newcommand{\linspace}[3][]{\mathcal{L}^{#1}\bigl({#2},{#3}\bigr)}

\newcommand{\radon}{\mathcal{M}}
\newcommand{\radonspace}[1]{\mathcal{M}({#1})}

\newcommand{\symgrad}{\mathcal{E}}

\newlength{\formulaindentwidth}

\begin{document}\sloppy

\def\sampta{SampTA~}

\title{Inverse problems with 
  second-order Total Generalized Variation constraints}

\name{Kristian Bredies, Tuomo Valkonen 
  \thanks{Supported by the Austrian Science Fund (FWF)
  under grant SFB F32 ``Mathematical Optimization and Applications 
  in Biomedical Sciences''.}}
\address{Institute of Mathematics and Scientific Computing, \\ 
  University of Graz, Graz, Austria \\
Emails: kristian.bredies@uni-graz.at, tuomo.valkonen@iki.fi}

\maketitle

\begin{abstract}
  Total Generalized Variation ($\TGV$) has recently been introduced as
  penalty functional for modelling images with edges as well as smooth 
  variations \cite{bredies2010tgv}. It can be interpreted as a ``sparse''
  penalization of optimal balancing from the first up to the $k$-th 
  distributional derivative and leads to desirable results when 
  applied to image denoising, i.e., $L^2$-fitting with $\TGV$ penalty.
  The present paper studies $\TGV$ of second order in the context
  of solving ill-posed linear inverse problems. Existence and stability
  for solutions of Tikhonov-functional minimization with respect
  to the data is shown and applied to the problem of recovering an image
  from blurred and noisy data.
\end{abstract}


\section{Introduction}
\label{sec:intro}

Most mathematical formulations of inverse problems, in 
particular of mathematical imaging problems are cast in the form
of minimizing a Tikhonov functional, i.e.,
\[
\min_{u} \ F(u) + \alpha R(u)
\]
where $F$ represents the fidelity with respect to the measured data,
$R$ is a regularization functional and $\alpha > 0$ a 
parameter. With a linear, 
continuous and usually ill-posed forward operator $K$ as well as
possibly error-prone data $f$,
the data fidelity term is commonly chosen as
\[
F(u) = \frac{\norm{Ku - f}^2}{2}
\]
with a Hilbert-space norm $\norm{\placeholder}$. For imaging problems,
popular choices for the regularization functionals 
are one-homogeneous functionals, in particular, 
the Total Variation seminorm \cite{osher1992tv}
\[
R(u) = \int_\Omega \dd{\abs{Du}} = \norm[\radon]{Du}
\]
where $\abs{Du}$ denotes the variation-measure of the 
distributional derivative $Du$ which is a vector-valued Radon measure.
It allows for discontinuities to appear along hypersurfaces and therefore
yields a suitable model for images with edges. Unfortunately, for true data
containing smooth regions, Total Variation regularization tends to produce 
undesired 
piecewise constant solutions; a phenomenon which is known as the 
``staircasing effect'' \cite{nikolova2000homogeneity,ring2000structuraltv}. 
To overcome this problem, 
higher-order functionals have been proposed \cite{chambolle1997tv,
  chan2005fourthorder}, for instance the
weighted 
infimal convolution of the first- and second order total-variation, i.e.,
$\int_\Omega \dd{\abs{Du}}$ and $\int_\Omega 
\dd{\abs{D^2u}}$. With such regularizers, a reduction of the 
staircasing effect can be observed. However, it may still occur, usually
in the neighborhood of edges.

In \cite{bredies2010tgv}, the \emph{Total Generalized 
  Variation} ($\TGV$) of order $k$, defined as
\begin{multline}
  \TGV_\alpha^k(u) = \sup \ \Bigl\{\int_\Omega u \divergence^k v \dd{x} \ \Bigl| 
  \ v \in \Ccspace{k}{\Omega,\Sym^k(\RR^d)}, \\
  \norm[\infty]{\divergence^l v} \leq \alpha_l, \ l=0,\ldots,k-1 \Bigr\},
  \label{eq:tgv_def}
\end{multline}
has been proposed and analyzed. It constitutes
a new image model which can be interpreted to incorporate smoothness 
from the first up to the $k$-th derivative. Here, $\Sym^k(\RR^d)$ denotes
the space of symmetric tensors of order $k$ with arguments in $\RR^d$ and 
$\alpha_l > 0$ are fixed parameters. Choosing $k=1$ and $\alpha_0 = 1$ 
yields the usual Total Variation functional. It is immediate that 
in $\LPspace{2}{\Omega}$, the denoising problem which corresponds 
to $K = I$, $\TGV_\alpha^2$ as a regularization
term leads to a well-posed minimization problem for the Tikhonov functional. 
The numerical experiments carried out in \cite{bredies2010tgv} show 
that $\TGV_\alpha^2$ produces visually appealing results with almost no 
staircase effect present in the solution (see Figure~\ref{fig:denoising}).

\begin{figure}[t]
  \centering
  \begin{tabular}{cc}
    \includegraphics[width=0.45\linewidth]{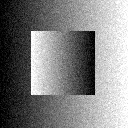} &
    \includegraphics[width=0.45\linewidth]{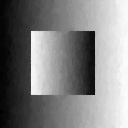} \\[-0.5ex]
    $f_{\mathrm{noise}}$ & $u_{\TV}$ \\[\smallskipamount]
    \includegraphics[width=0.45\linewidth]{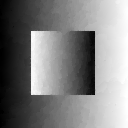} &
    \includegraphics[width=0.45\linewidth]{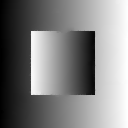} \\[-0.5ex]
    $u_{\mathrm{inf-conv}}$ & $u_{\TGV_\alpha^2}$
  \end{tabular}
  \caption{Denoising with Total Variation, infimal convolution and
    Total Generalized Variation.}
  \label{fig:denoising}
\end{figure}

When solving ill-posed inverse problems for 
$K: \LPspace{2}{\Omega} \rightarrow 
\LPspace{2}{\Omega}$ with a Total Generalized Variation 
regularization, however, 
it is not immediate that the problem of minimizing the 
Tikhonov functional is well-posed. The present paper addresses this 
issue by analyzing the case of $k=2$, i.e., Total 
Generalized Variation of second order. We will show that $\TGV_\alpha^2$
is a semi-norm on $\BV(\Omega)$ and that $\norm[1]{u} + \TGV_\alpha^2(u)$
is topologically equivalent to the $\BV$-norm. 
Based on this result, existence and
well-posedness of $\TGV_\alpha^2$-regularization for Tikhonov functionals
is derived. Moreover, we apply these results to solve an image deblurring 
problem. Finally, numerical experiments illustrate the feasibility of 
Total Generalized Variation regularization of second order.

\section{Basic properties of second-order TGV}

Let us first define Total Generalized Variation of second order 
as well as mention some basic properties.

\begin{definition}
  Let $\Omega \subset \RR^d$ be a bounded domain 
  and $\alpha = (\alpha_0, \alpha_1) 
  > 0$.
  The functional assigning each $u \in \LPlocspace{1}{\Omega}$
  the value
  \begin{multline}
    \label{eq:tgv2_def}
    \TGV_\alpha^2(u) = \sup\ \Bigl\{\int_\Omega u \divergence^2 v \dd{x} \ \Bigl| \ 
    v \in \Ccspace{2}{\Omega,S^{d \times d}}, \\ 
    \norm[\infty]{v} \leq \alpha_0, \ \norm[\infty]{\divergence v} 
    \leq \alpha_1 
    \Bigr\}
  \end{multline}
  is called the \emph{Total Generalized Variation} of second order.
  
  Here, $S^{d \times d}$ is the set of symmetric matrices, $\Ccspace{2}{\Omega,
    S^{d \times d}}$ the vector space of compactly supported, twice continuously 
  differentiable
  $S^{d \times d}$-valued mappings and $\divergence v \in 
  \Ccspace{1}{\Omega,\RR^d}$, $\divergence^2 v \in \Ccspace{}{\Omega}$ 
  is defined by
  \begin{align*}
    (\divergence v)_i &= \sum_{j=1}^d \frac{\partial v_{ij}}{\partial x_j}, & 
    \divergence^2 v &= \sum_{i=1}^d \frac{\partial^2 v_{ii}}{\partial x_i^2} 
    + 2 \sum_{i < j} \frac{\partial^2 v_{ij}}{\partial x_i \partial x_j}.
  \end{align*}
  The norms of $v \in \Ccspace{}{\Omega,S^{d \times d}}$, $\omega 
  \in \Ccspace{}{\Omega, \RR^d}$ are given by
  \begin{align*}
    \norm[\infty]{v} &= \sup_{x \in \Omega} \ \Bigl( \sum_{i=1}^d \abs{v_{ii}(x)}^2 + 2 \sum_{i < j} \abs{v_{ij}(x)}^2 \Bigr)^{1/2}, \\
    \norm[\infty]{\omega} &= \sup_{x \in \Omega} \ \Bigl( \sum_{i=1}^d \abs{\omega_i(x)}^2 \Bigr)^{1/2}.
  \end{align*}
  The space
  \[
  \BGV_\alpha^2(\Omega) = \set{u \in \LPspace{1}{\Omega}}{\TGV_\alpha^2(u) 
    < \infty}
  \]
  equipped with the norm
  \[
  \norm[\BGV_\alpha^2]{u} = \norm[1]{u} + \TGV_\alpha^2(u)
  \]
  is called the space of functions of \emph{Bounded Generalized Variation}
  of order $2$.
\end{definition}

Basic results about this functional obtained in  \cite{bredies2010tgv} 
can be summarized as follows.

\begin{theorem}
  \label{thm:tgv_prop}
  Total Generalized Variation of second order enjoys the following properties:
  \begin{compactenum}
  \item $\TGV_\alpha^2$ is a semi-norm on the Banach space 
    $\BGV_\alpha^2(\Omega)$,
  \item $\TGV_\alpha^2(u) = 0$ if and only if $u$ is a polynomial of degree 
    less than $2$,
  \item
    $\TGV_\alpha^2$ and $\TGV_{\tilde \alpha}^2$ are equivalent for $\tilde \alpha = 
    (\tilde\alpha_0, \tilde \alpha_1) > 0$,
  \item
    $\TGV_\alpha^2$ is rotationally invariant,
  \item
    $\TGV_\alpha^2$ satisfies, for $r > 0$ and $(\rho_ru)(x) = u(rx)$, 
    the scaling property
    \[
    \TGV_\alpha^2 \compose \rho_r = r^{-d} \TGV_{\tilde \alpha}^2(u), \ \ 
    (\tilde\alpha_0, \tilde \alpha_1) = (\alpha_0r^2, \alpha_1r),
    \]
  \item
    $\TGV_\alpha^2$ is proper, convex and lower semi-continuous on each 
    $\LPspace{p}{\Omega}$, $1 \leq p < \infty$.
  \end{compactenum}
\end{theorem}

\section{Topological equivalence with BV}

To obtain existence results for Tikhonov functionals with 
$\TGV_\alpha^2$-penalty, we reduce the functional-analytic setting to
the space $\BV(\Omega)$ by establishing topological equivalence. 
This will be done in two steps. 

\begin{theorem}
  Let $\Omega \subset \RR^d$ be a bounded domain.
  For each $u \in \LPspace{1}{\Omega}$ we have
  \begin{equation}
    \label{eq:tgv2_min}
    \TGV_\alpha^2(u) = \min_{w \in \BD(\Omega)} \ \alpha_1\norm[\radon]{Du - w}
    + \alpha_0 \norm[\radon]{\symgrad w}
  \end{equation}
  where $\BD(\Omega)$ denotes the space of vector fields of \emph{Bounded
    Deformation} \cite{temam85plasticity}, 
  i.e., $w \in \LPspace{1}{\Omega,\RR^d}$ such that
  the distributional symmetrized derivative $\symgrad w = \frac12(\grad w 
  + \grad w^\transp)$ is a $S^{d \times d}$-valued Radon measure.
\end{theorem}

\begin{proof}[Sketch of proof]
  Choosing $X = \Cspace[0]{2}{\Omega,S^{d\times d}}$, 
  $Y = \Cspace[0]{1}{\Omega,\RR^d}$, $\Lambda = \divergence 
  \in \linspace{X}{Y}$ and, for $v \in X$, $\omega \in Y$,
  \begin{align*}
    F_1(v) 
    &= I_{\sett{\norm[\infty]{\placeholder} \leq \alpha_0}}(v), \\
    F_2(\omega) &=
    I_{\sett{\norm[\infty]{\placeholder} \leq \alpha_1}}(\omega) 
    - \int_\Omega u 
    \divergence \omega \dd{x}
  \end{align*}
  we see with density arguments that
  \[
  \TGV_\alpha^2(u) = - \inf_{v \in X} \ F_1(v) + F_2(\Lambda v).
  \]
  Furthermore, $Y = \bigcup_{\lambda \geq 0} \lambda \bigl( \dom(F_1) - 
  \Lambda \dom(F_2) \bigr)$, hence by Fenchel-Rockafellar duality 
  \cite{attouch1986duality}, it follows that
  \[
  \TGV_\alpha^2(u) = \min_{w \in Y^*} \ F_1^*(-\Lambda^*w) + F_2^*(w).
  \]
  Now, $Y^* = \Cspace[0]{1}{\Omega,\RR^d}^*$ can be regarded as a space
  of distributions and the dual functionals can be written as
  \begin{align*}
    F_1^*(-\Lambda^*w) &=
    \begin{cases}
      \alpha_0 \norm[\radon]{\symgrad w} & \text{if} \ 
      w \in \BD(\Omega) \\
      \infty & \text{else},
    \end{cases} \\
    F_2^*(w) &=
    \begin{cases}
      \alpha_1 \norm[\radon]{Du - w} & \text{if} \ Du - w \in 
      \radonspace{\Omega,\RR^d} \\
      \infty & \text{else}.
    \end{cases}
  \end{align*}
  Since $\BD(\Omega) \subset \radonspace{\Omega,\RR^d}$, the result follows.
\end{proof}

\begin{remark}
  The minimization in~\eqref{eq:tgv2_min} can be interpreted as an
  optimal balancing between the first and second derivative of $u$ in 
  terms of ``sparse'' penalization (via the Radon norm).
\end{remark}

The second step combines~\eqref{eq:tgv2_min} with the 
Sobolev-Korn inequality for vector fields of Bounded Deformation.

\begin{theorem}
  \label{thm:tgv_bg_equiv}
  Let $\Omega \subset \RR^d$ be a bounded Lipschitz domain. Then there
  exist constants $0 < c < C < \infty$ such that for each $u \in 
  \BGV_\alpha^2(\Omega)$ there holds
  \[
  c \norm[\BV]{u} \leq \norm[1]{u} + \TGV_\alpha^2(u) \leq C \norm[\BV]{u}.
  \]
\end{theorem}

\begin{proof}
  Setting $w = 0$ in~\eqref{eq:tgv2_min}
  immediately implies that for each $u \in \BGV_\alpha^2(\Omega)$ we have
  $\TGV_\alpha^2(u) \leq \alpha_1 \TV(u)$, hence
  we can set $C = \max(1,\alpha_1)$.
  
  On the other hand, we may assume that $Du \in 
  \radonspace{\Omega,\RR^d}$ since otherwise, $\norm[\radon]{Du - w}
  = \infty$ for all 
  $w \in \BD(\Omega)$ and hence,
  $\TGV_\alpha^2(u) = \infty$ by~\eqref{eq:tgv2_min}.
  Observe that $\bar w \in \ker{\symgrad}$ if and only if
  $\bar w(x) = Ax + b$ for some $A \in \RR^{d \times d}$ satisfying $A^\transp 
  = -A$ and $b \in \RR^d$.
  We  show that there is a $C_1 > 0$ such that 
  for each $u \in \BV(\Omega)$ and $\bar w \in \ker \symgrad$ there holds
  \begin{equation}
    \label{eq:tgv_equiv_help}
    \norm[\radon]{Du} \leq C_1 \bigl( \norm[\radon]{Du 
      - \bar w} + \norm[1]{u} \bigr).
  \end{equation}
  Suppose that this is not the case. Then, there exist sequences
  $\seq{u^n}$ in 
  $\BV(\Omega)$ and $\seq{\bar w^n}$ in $\ker \symgrad$ such that 
  \[
  \norm[\radon]{Du^n} = 1, \quad
  \norm[\radon]{Du^n - \bar w^n} + \norm[1]{u^n} \leq \tfrac1n
  \]
  for all $n \in \NN$.
  Consequently, $\lim_{n\rightarrow \infty} u^n = 0$ in $\LPspace{1}{\Omega}$,
  $\lim_{n\rightarrow \infty} Du^n - \bar w^n  = 0$ 
  in $\radonspace{\Omega,\RR^d}$  
  and $\seq{\bar w^n}$ is bounded in $\ker\symgrad$. Since the latter is 
  finite-dimensional, there exists a convergent subsequence, i.e., 
  $\lim_{k\rightarrow \infty} \bar w^{n_k} = w$ for some $w \in \ker\symgrad$.
  It follows that $\lim_{k \rightarrow \infty} Du^{n_k} = w$, thus $w = 0$ by 
  closedness of the distributional derivative. This means in particular
  $\lim_{k \rightarrow \infty} \norm[\radon]{Du^{n_k}} = 0$ which is
  a contradiction since each $\norm[\radon]{Du^{n_k}} = 1$.
  
  Next, recall that a Sobolev-Korn inequality holds for $\BD(\Omega)$ 
  \cite{temam85plasticity}:
  There is a $C_2 > 0$ such that for each $w \in \BD(\Omega)$ there
  exists a $\bar w \in \ker\symgrad$ such that $\norm[1]{w - \bar w} 
  \leq C_2 \norm[\radon]{\symgrad w}$. For this $\bar w$, we have, 
  for some $C_3 > 0$,
  \begin{align*}
    \norm[\radon]{Du - \bar w} &\leq \norm[\radon]{Du - w} 
    + \norm[1]{w - \bar w} \\
    &\leq C_3 \bigl( \alpha_1 \norm[\radon]{Du - w} 
    + \alpha_0 \norm[\radon]{\symgrad w} \bigr).
  \end{align*}
  Plugged into~\eqref{eq:tgv_equiv_help} and adding $\norm[1]{u}$ 
  on both sides, it follows that 
  the inequality 
  \[
  \norm[\BV]{u} 
  \leq C_4 \bigl(\norm[1]{u} + \alpha_1 \norm[\radon]{Du - w} + 
  \alpha_0 \norm[\radon]{\symgrad w} \bigr)
  \]
  holds for some $C_4> 0$ independent of $u$ and $w$. Taking the minimum
  over all $w \in \BD(\Omega)$ and choosing $c = C_4^{-1}$ finally yields 
  the result by virtue of~\eqref{eq:tgv2_min}.
\end{proof}

Since both $\BV(\Omega)$ and $\BGV_\alpha^2(\Omega)$ are Banach spaces, we 
immediately have:

\begin{corollary}
  If $\Omega \subset\RR^d$ is a bounded Lipschitz-domain, then
  $\BGV_\alpha^2(\Omega) = \BV(\Omega)$ for all $(\alpha_0,\alpha_1) > 0$
  in the sense of topologically equivalent Banach spaces.
\end{corollary}

\section{Existence and stability of solutions}

Let, in the following, $\Omega \subset \RR^d$ be a bounded
Lipschitz domain. The coercivity needed for showing existence of solutions
for the Tikhonov functional is implied by the following inequality
of Poincar\'e-Wirtinger type.

\begin{proposition}
  \label{prop:tgv_poincare}
  Let $1 < p < \infty$ such that $p \leq d/(d-1)$ and 
  $P:\LPspace{p}{\Omega} \rightarrow \mathcal{P}^1(\Omega)$ a linear
  projection onto the space of affine functions $\mathcal{P}^1(\Omega)$.
  Then, there is a $C > 0$ such that
  \begin{equation}
    \label{eq:tgv_poincare}
    \norm[p]{u} \leq C \TGV_\alpha^2(u) \qquad \forall u \in 
    \ker P \subset \LPspace{p}{\Omega}.
  \end{equation}
\end{proposition}

\begin{proof}
  If this is not true, there exists a 
  sequence $\seq{u^n}$ in $\ker P$ with 
  $\norm[p]{u^n} = 1$ such that 
  $1 \geq  C(n) \TGV_\alpha^2(u^n)$ where $C(n) \geq n$. 
  We may assume $u^n \wrightarrow u$ in $\LPspace{p}{\Omega}$ 
  with $u \in \ker P$. According to Theorem~\ref{thm:tgv_bg_equiv}, 
  $\seq{u^n}$ is also bounded in $\BV(\Omega)$, thus we also have, 
  by compact embedding, that $\lim_{n \rightarrow \infty} u^n = u$ 
  in $\LPspace{1}{\Omega}$. Lower semi-continuity now implies 
  $\TGV_\alpha^2(u) \leq \liminf_{n\rightarrow \infty} \TGV_\alpha^2(u^n) = 0$, hence 
  $u \in \mathcal{P}^1(\Omega) \cap \ker P$ (see Theorem~\ref{thm:tgv_prop})
  and, consequently, $u = 0$. Thus, $u^n \rightarrow 0$ in 
  $\BV(\Omega)$ and by continuous embedding, also in $\LPspace{p}{\Omega}$,
  which is a contradiction to $\norm[p]{u^n} = 1$ for all $n$.
\end{proof}

\begin{theorem}
  \label{thm:tgv_tikhonov}
  Let $1 < p < \infty$, $p \leq d/(d-1)$, 
  $Y$ be a Hilbert space, $K \in
  \linspace{\LPspace{p}{\Omega}}{Y}$ a linear and continuous 
  operator which is injective on $\mathcal{P}^1(\Omega)$ and $f \in Y$.
  Then, the problem
  \begin{equation}
    \label{eq:tgv_tikhonov}
    \min_{u \in \LPspace{p}{\Omega}} \ \tfrac12\norm{Ku - f}^2 + \TGV_\alpha^2(u)
  \end{equation}
  admits a solution.
\end{theorem}

\begin{proof}
  Choose $P$ according to
  Proposition~\ref{prop:tgv_poincare} (such a $P$ exists since 
  $\mathcal{P}^1(\Omega)$ is finite-dimensional) as well as a
  minimizing sequence $\seq{u^n}$. By $\TGV_\alpha^2(u) = 
  \TGV_\alpha^2(u -Pu)$ (since $\TGV_\alpha^2(Pu) = 0$, see 
  Theorem~\ref{thm:tgv_prop}) and~\eqref{eq:tgv_poincare}, 
  $\seq{u^n - Pu^n}$ is bounded in $\LPspace{p}{\Omega}$.
  Moreover, $\seq{\norm{Ku^n - f}}$ is bounded.
  Since $K$ is injective on the finite-dimensional space
  $\mathcal{P}^1(\Omega)$, there is a $C_1 > 0$
  such that $\norm[p]{Pu} \leq C_1 \norm{KPu}$, hence
  \begin{multline*}
    \norm[p]{Pu^n} \leq C_1 \norm{KPu^n} \\ 
    \leq C_1 \bigl( \norm{Ku^n - f} + 
    \norm{K(u^n - Pu^n) - f} \bigr) \leq C_2,
  \end{multline*}
  for some $C_2 > 0$
  implying that $\seq{u^n}$ is bounded in $\LPspace{p}{\Omega}$.
  Thus, there exists a weakly convergent subsequence with limit $u^*$ 
  which can be seen to be a minimizer by weak lower semi-continuity (also
  confer Theorem~\ref{thm:tgv_prop}).
\end{proof}

\begin{theorem}
  \label{thm:tgv_tikh_stab}
  In the situation of Theorem~\ref{thm:tgv_tikhonov}, let $\seq{f^n}$
  be a sequence in $Y$ with $\lim_{n\rightarrow \infty} f^n = f$.
  Then each sequence $\seq{u^n}$ of minimizers $u^n$ 
  of~\eqref{eq:tgv_tikhonov} with data $f^n$ is relatively 
  weakly compact in $\LPspace{p}{\Omega}$ and each weak accumulation
  point $u^* = \lim_{k \rightarrow \infty} u^{n_k}$ 
  minimizes~\eqref{eq:tgv_tikhonov} with data $f$ with
  $\lim_{k \rightarrow \infty} = \TGV_\alpha^2(u^{n_k}) = \TGV_\alpha^2(u^*)$.
\end{theorem}

\begin{proof}
  The proof follows the lines of \cite{hofmann2007convergencerates}. Denote by $F_n$ and $F$ the 
  functional to minimize in~\eqref{eq:tgv_tikhonov} with data $f^n$ and $f$, 
  respectively. Then, for any $u \in \BV(\Omega)$, $F_n(u^n) 
  \leq F_n(u)$, hence, using that $\norm{a+b}^2 \leq 2 \norm{a}^2 + 
  2\norm{b}^2$ for $a,b\in Y$,
  \begin{multline*}
    F(u^n) \leq \norm{Ku^n - f^n}^2 + \TGV_\alpha^2(u^n) + \norm{f^n - f}^2 \\
    \leq 2F_n(u^n) + \norm{f^n - f}^2 \leq 2F_n(u) + \norm{f^n - f}^2.
  \end{multline*}
  This shows that $\seq{F(u^n)}$ is bounded and by the same arguments as
  in Theorem~\ref{thm:tgv_tikhonov}, it follows that $\seq{u^n}$ is 
  bounded in $\LPspace{p}{\Omega}$ and therefore relatively weakly compact.
  Now, let $u^*$ be a weak accumulation point, i.e., $u^{n_k} \wrightarrow 
  u^*$. It follows by weak lower semi-continuity that $\frac12 
  \norm{Ku^* - f}^2 \leq \liminf_{k \rightarrow \infty} \frac12 \norm{Ku^{n_k} - 
    f^{n_k}}^2$
  and $\TGV_\alpha^2(u^*) \leq \liminf_{k \rightarrow \infty} \TGV_\alpha^2(u^{n_k})$.
  Hence,
  \begin{align*}
    F(u^*) &\leq \liminf_{k \rightarrow \infty} F_{n_k}(u^{n_k}) 
    \leq 
    \limsup_{k \rightarrow \infty} F_{n_k}(u^{n_k}) \\
    &\leq
    \lim_{k \rightarrow \infty} F_{n_k}(u) = F(u)
  \end{align*}
  for each $u \in \BV(\Omega)$,
  showing that $u^*$ is a minimizer. In particular, plugging in $u^*$
  leads to $\lim_{k \rightarrow \infty} F_{n_k}(u^{n_k}) = F(u^*)$. Now, 
  $\limsup_{k \rightarrow \infty} \TGV_\alpha^2(u^{n_k}) > \TGV_\alpha^2(u^*)$
  would contradict this, hence $\TGV_\alpha^2(u^{n_k}) \rightarrow 
  \TGV_\alpha^2(u^*)$.
\end{proof}

\section{Application to deconvolution}

Let $d = 2$ and $\Omega \subset \RR^2$ be a bounded Lipschitz domain.
Pick a blurring kernel $k \in L^1(\Omega_0)$ satisfying 
$\bar k = \int_{\Omega_0} k \dd{x} \neq 0$. Moreover, let 
$\Omega' \subset \RR^2$
be a domain with $\Omega' - \Omega_0 \subset \Omega$.
Then
\[
(Ku)(x) = \int_{\Omega_0} u(x-y)k(y) \dd{y}, \quad x \in \Omega'
\]
fulfills $K \in \linspace{\LPspace{2}{\Omega}}{\LPspace{2}{\Omega'}}$.
If $m$ denotes the vector with components
$m_i = \int_{\Omega_0} y_i k(y) \dd{y}$, then 
an affine function $u(x) = a\inprod x 
+ b$ with $Ku = 0$ satisfies
\[
\int_{\Omega_0} \bigl (a \inprod (x-y) + b \bigr) 
k(y) \dd{y} = a\bar k \inprod x + b\bar k  - a \inprod m = 0
\]
for all $x \in \Omega'$. Since this is a domain, $a \bar k = 0$ and 
$b \bar k - a \inprod m = 0$ which implies $a = 0$ and $b = 0$. Hence
$K$ is injective on $\mathcal{P}^1(\Omega)$.

If $f \in \LPspace{2}{\Omega'}$ is a noise-contaminated image blurred by
the convolution operator $K$, then according to 
Theorem~\ref{thm:tgv_tikhonov}, the $\TGV_\alpha^2$-based
Tikhonov regularization
\[
\min_{u \in \LPspace{2}{\Omega}} \ \frac12 \int_{\Omega'} 
\abs{(u \conv k)(x) - f(x)}^2 \dd{x} + \TGV_\alpha^2(u) 
\]
admits a solution which stably depends on $f$ 
(Theorem~\ref{thm:tgv_tikh_stab}).

This convex minimization problem can be put, for instance, 
in a saddle-point formulation and solved numerically by a 
primal-dual algorithm \cite{chambolle2010primaldual}. 
Figure~\ref{fig:tgv_deconv}
shows the effect of $\TGV_\alpha^2$ compared to $\TV$ for deblurring
a sample image.

\begin{figure}
  \centering
  \begin{tabular}{cc}
    \includegraphics[width=0.45\linewidth]{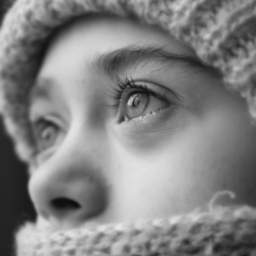} &
    \includegraphics[width=0.45\linewidth]{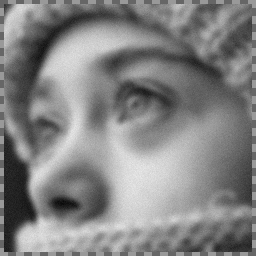} \\[-0.5ex]
    $u_{\mathrm{orig}}$ & $f$ \\[\smallskipamount]
    \includegraphics[width=0.45\linewidth]{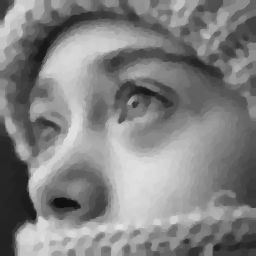} &
    \includegraphics[width=0.45\linewidth]{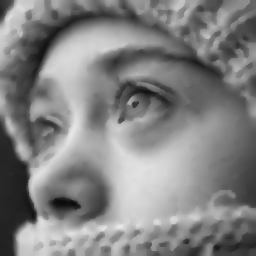} \\[-0.5ex]
    $u_{\TV}$ 
    & $u_{\TGV_\alpha^2}$ 
  \end{tabular}
  \caption{Deconvolution example. The original image $u_\mathrm{orig}$ 
    \cite{alina2009eye} has been blurred and contaminated by 
    noise resulting in $f$, $u_{\TV}$ 
    and $u_{\TGV_\alpha^2}$ are the regularized solutions recovered from $f$.}
  \label{fig:tgv_deconv}
\end{figure}

\bibliographystyle{plain}
\bibliography{myRef}

\end{document}